\documentclass[a4paper,abstracton]{scrartcl}
\usepackage{amsthm,amsmath,amssymb}
\usepackage{hyperref}
\usepackage{enumitem}
\usepackage{authblk}

\usepackage{tabularx}
\newcolumntype{Y}{>{\centering\arraybackslash}X}
\usepackage{pgf,tikz,ifthen}
\usetikzlibrary{math,positioning}

\newtheorem{theorem}{Theorem}[section]
\newtheorem{corollary}[theorem]{Corollary}

\newtheorem{lemma}[theorem]{Lemma}

\theoremstyle{definition}
\newtheorem{question}[theorem]{Question}

\newtheorem{example}[theorem]{Example}

\newtheorem{problem}[theorem]{Problem}

\def\Aut{\operatorname{Aut}}

\title{Distinguishing numbers of finite $4$-valent vertex-transitive graphs}
\author{Florian Lehner\thanks{Supported by the Austrian Science Fund (FWF), grant J 3850-N32}\, and Gabriel Verret}
\begin{document}

\maketitle

\begin{abstract}
The distinguishing number of a graph $G$ is the smallest $k$ such that $G$ admits a $k$-colouring for which the only colour-preserving automorphism of $G$ is the identity. We determine the distinguishing number of finite $4$-valent vertex-transitive graphs. We show that, apart from one infinite family and finitely many examples, they all have distinguishing number $2$.
\end{abstract}

\section{Introduction}
All graphs in this paper will be finite. A \emph{distinguishing colouring} of a graph is a colouring which is not preserved by any non-identity automorphism. The \emph{distinguishing number} $D(G)$ of a graph $G$ is the least number of colours needed for a distinguishing colouring of the vertices of $G$. These concepts were first introduced by Albertson and Collins \cite{albertsoncollins} and have since received considerable attention.

It is an easy observation that a graph has distinguishing number $1$ if and only if its automorphism group is trivial. Hence, by \cite{erdosrenyi} almost all graphs have distinguishing number $1$. This obviously is not true for vertex-transitive graphs which always have non-trivial automorphisms. However, it seems that the vast majority of vertex-transitive graphs still have the lowest possible distinguishing number $2$. Hence let us call a vertex-transitive graph \emph{exceptional} if its distinguishing number is not equal to $2$.

One of the most interesting results concerning distinguishing numbers of vertex-transitive graphs is that, apart from the complete and edgeless graphs, there are only finitely many exceptional vertex-primitive graphs~\cite{Cameron,seress}. It is only natural to ask whether something similar holds for vertex-transitive graphs as well. As a first step, H\"uning et al.\ recently determined the exceptional $3$-valent vertex-transitive graphs and their distinguishing numbers.

\begin{theorem}\cite[Corollary~2.2]{HuningEtAl}\label{theo:cubic}
The exceptional connected $3$-valent vertex-transitive graphs are 
\begin{enumerate}
\item $K_4$ and $K_{3,3}$, with distinguishing number $4$, and
\item $Q_3\cong K_4\times K_2$ and the Petersen graph, with distinguishing number $3$.
\end{enumerate}
\end{theorem}

This result shows that there are only finitely many connected $3$-valent vertex-transitive exceptional graphs. This is not true for $4$-valent graphs, as shown by the following family of graphs. For $n\geq 3$, the \emph{wreath graph $W_n$} is the lexicographic product $C_n[2K_1]$ of a cycle of length $n$ with an edgeless graph of order $2$, see Figure \ref{fig:wreath}.

\begin{figure}[hh]
\centering
\begin{tikzpicture}
	\pgfdeclarelayer{E}
	\pgfdeclarelayer{V}
	\pgfsetlayers{E,V}
	\tikzmath{\i=.9; \o=1.5; \n=10; \s=360/\n;}
	\foreach \a in {1,...,\n}
	{
    \begin{pgfonlayer}{V}
		\node[inner sep=1.5pt,circle,draw,fill] (v\a) at ({\a*\s}:\o){};
		\node[inner sep=1.5pt,circle,draw,fill] (w\a) at ({\a*\s}:\i){};
    \end{pgfonlayer}
    
    \begin{pgfonlayer}{E}
        \path[draw] ({\a*\s}:\o)--({\s+\a*\s}:\o);
        \path[draw] ({\a*\s}:\o)--({\s+\a*\s}:\i);
        \path[draw] ({\a*\s}:\i)--({\s+\a*\s}:\o);
        \path[draw] ({\a*\s}:\i)--({\s+\a*\s}:\i);
    \end{pgfonlayer}
	}
\end{tikzpicture}
\caption{The wreath graph $W_{10}$}
\label{fig:wreath}
\end{figure}
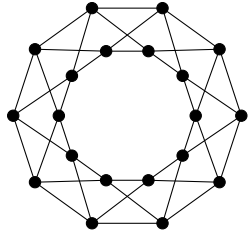

It is easy to see that wreath graphs  form an infinite family of connected exceptional $4$-valent vertex-transitive graphs, thus providing a negative answer to~\cite[Question~2]{HuningEtAl}. Our main result shows that this is the only such family, that is, apart from the wreath graphs, there are only finitely many connected exceptional $4$-valent vertex-transitive graphs.

\begin{theorem}\label{thm:main}
The exceptional connected  $4$-valent vertex-transitive graphs are  
\begin{enumerate}
\item $K_5$ and $K_{4,4} \cong W_4$, with distinguishing number $5$, and 
\item $K_3 \square K_3$, $K_4 \square K_2$, $K_5\times K_2$ and $W_n$ for some $n\geq 3$, $n\neq 4$, with distinguishing number $3$.
\end{enumerate}
\end{theorem}

In particular, there is no example with distinguishing number $4$. This leads us to the following question.

\begin{question}
For $\Delta\geq 5$, is there a connected $\Delta$-valent vertex-transitive  graph $G$ with $D(G)=\Delta$? 
\end{question}

More generally, one could ask about ``gaps'' in the set of distinguishing numbers of connected $\Delta$-valent vertex-transitive graphs, as a subset of $\{2,\ldots,\Delta+1\}$.

Using lexicographic products, it is not hard to construct infinite families of connected exceptional vertex-transitive graphs with fixed valency.
\begin{example}
 Let $H_1$ be a connected vertex-transitive graph of valency $\Delta_1$ and let $H_2$ be a vertex-transitive graph of valency $\Delta_2$ on $n_2$ vertices. Then the lexicographic product $H_1[H_2]$ is connected, has valency $\Delta_1 n_2 + \Delta_2$ and its distinguishing number is at least $D(H_2)+1$. For an infinite family of examples that are not lexicographic products, note that, for every $n\geq 3$ and every $d\geq 2$, the graph $(C_n[d^2K_1])\square K_2$ has valency $2d^2+1$ and distinguishing number strictly greater than $d$. 
\end{example}

We hence pose the following (informal) problem.
\begin{problem}
Is there a  ``natural small family'' $\mathcal{F}$ of exceptional graphs such that, for every positive integer $k$, all but finitely many $k$-valent connected exceptional vertex-transitive graphs are contained in $\mathcal{F}$?
\end{problem}

\section{Definitions and auxiliary results}

Throughout this paper, all graphs are assumed to be finite and simple. Graph theoretic notions that are not explicitly defined will be taken from \cite{diestelbook}.

An \emph{automorphism} of a graph is an adjacency preserving permutation of its vertices. The group of all automorphisms of  a graph $G$ is denoted by $\Aut G$. We say that a graph is \emph{vertex-transitive} if its automorphism group is transitive (that is, for every pair of vertices, there exists an automorphism mapping the first to the second). 

An \emph{arc} in a graph $G$ is an ordered pair of adjacent vertices, or equivalently, a walk of length $2$ in $G$. An \emph{$s$-arc} is a non-backtracking walk of length $s$ in $G$, i.e.\ a sequence of vertices $v_0, \dots, v_s$ where $v_i$ is adjacent to $v_{i+1}$ for $0 \leq i \leq s-1$, and $v_{i-1} \neq v_{i+1}$ for $1 \leq i \leq s-1$. The automorphism group $\Aut G$ acts on the set of edges, arcs, and $s$-arcs of $G$ in an obvious way. Call a graph \emph{edge-transitive}, \emph{arc-transitive}, or \emph{$s$-arc-transitive}, if the action of $\Aut G$ on edges, arcs, or $s$-arcs is transitive, respectively. Analogously define  \emph{arc-regular} and \emph{$s$-arc-regular}. 

The \emph{local group} at a vertex $v$ is the permutation group induced by the stabiliser of $v$ acting on its neighbourhood $N(v)$. Note that, for vertex-transitive graphs, this does not depend on the choice of $v$ (up to permutation equivalence). We say that a graph is locally $\Gamma$, if the local group is isomorphic to $\Gamma$. 

A graph $G$ is called \emph{$k$-connected} if it remains connected after removing any set of at most $k-1$ vertices and all incident edges, and \emph{$k$-edge connected} if it remains connected after removing any set of $k$ edges. The following result about the connectivity of vertex-transitive graphs is due to Watkins~\cite{connectivity}.

\begin{lemma}
\label{lem:connectivity}
A vertex-transitive graph with valency $r$ is at least $\frac{2r}3$-connected.
\end{lemma}

If we impose additional properties on the set of vertices to be removed, then we can remove much larger sets without disconnecting the graph. The following lemma follows easily from results in \cite{cyclicconnectivity}.

\begin{lemma}
\label{lem:cyclicconnectivity}
If $G$ is a $k$-valent vertex-transitive graph with $k \geq 4$ and girth $g \geq 5$, then there is a $g$-cycle $C$ in $G$ such that $G-C$ is $2$-edge connected.
\end{lemma}

\begin{proof}
By \cite[Theorem 4.5]{cyclicconnectivity}, there is a $g$-cycle $C$ such that the edges with one endpoint in $C$ and the other endpoint in $H:= G - C$ form a minimum (w.r.t.\ cardinality) cut separating two cycles in $G$. Assume that $H$ was not $2$-edge connected and let $e$ be a cut-edge of $H$. Let $A$ and $B$ be the two components of $H - e$. By \cite[Lemma 3.3]{cyclicconnectivity}, the minimum degree of $H$ is $2$, so $A$ and $B$ each contain at most one vertex of degree $1$, and thus there are cycles in both components. Now either the cut separating $A \cup C$ from $B$, or the cut separating $B \cup C$ from $A$ contains strictly fewer edges than the cut separating $C$ from $H$, contradicting the minimality. 
\end{proof}

We will also need the notion of \emph{distinguishing index} $D'(G)$ of a graph $G$, which is the least number of colours needed for a distinguishing colouring of the edges of $G$. Here are a few  results giving upper bounds on $D'(G)$. The first two are Theorems~2.8 and 3.2 in~\cite{Monika}.

\begin{theorem}
\label{monikabound}
Let $G$ be a connected graph that is neither a symmetric nor a bisymmetric tree. If the maximum degree $\Delta(G)$ of $G$ is at least 3, then $D'(G)\leq \Delta(G) - 1$ unless G is $K_4$ or $K_{3,3}$.
\end{theorem}

\begin{theorem}
\label{traceable}
If $G$ is a graph of order at least $7$ with a Hamiltonian path, then $D'(G)\leq 2$.
\end{theorem}

\begin{lemma}\label{lem:linegraph}
If $G$ is a connected graph on $5$ or more vertices, then $\Aut L(G)$ is permutationally equivalent to $\Aut G$ with its natural action on $E(G)$. Furthermore, in this case $D'(G) \leq D(G)$, unless $G$ is a tree.
\end{lemma}
\begin{proof}
 The first part is a variant of Whitney's theorem due to Jung \cite{junglinegraph}, the second part follows by applying \cite[Theorem 1.3]{lehnersmith} to a distinguishing colouring with $D(G)$ colours.
\end{proof}

In the remainder of this section, we discuss some known results on distinguishing numbers and determine the distinguishing numbers of several graphs that will occur in the proof of Theorem~\ref{thm:main}. The following lemma gives a general bound on distinguishing numbers and was independently proved in \cite{deltabound1} and \cite{deltabound2}.

\begin{lemma}
\label{degreebound}
 If $G$ is a connected graph with maximum degree $\Delta$, then $D(G) \leq \Delta + 1$, with equality if and only if $G$ is either $C_5$, or $K_n$ or $K_{n,n}$ for some $n \geq 1$.
\end{lemma}

For $n\geq 2$, we  define a family of graphs $C_{n,K_{3,3}}$ as follows. For $1 \leq i \leq n$, let $H_i$ be disjoint copies of $K_{3,3}$ with bipartition $V(H_i) = X_i \cup Y_i$. Let $C_{n,K_{3,3}}$ be the graph obtained from this collection by adding a matching between $X_i$ and $Y_{i+1}$ for $1 \leq i \leq n-1$, and between $X_n$ and $Y_1$, see Figure \ref{fig:k33cycle}.

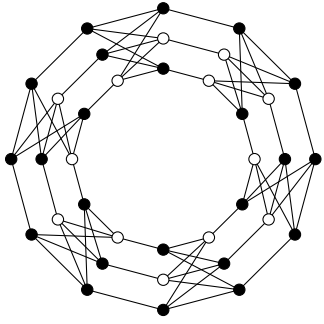
\begin{figure}
\centering
\begin{tikzpicture}
	\pgfdeclarelayer{E}
	\pgfdeclarelayer{V}
	\pgfsetlayers{E,V}
	\tikzmath{\i=1.2; \m=1.6; \o=2; \n=6; \s=180/\n;}
    \foreach \a in {1}
	{
    \begin{pgfonlayer}{V}
		\node[inner sep=1.5pt,circle,draw,fill=white] at ({2*\a*\s}:\i){};
		\node[inner sep=1.5pt,circle,draw,fill=white] at ({2*\a*\s}:\m){};
		\node[inner sep=1.5pt,circle,draw,fill] at ({2*\a*\s}:\o){};
		\node[inner sep=1.5pt,circle,draw,fill] at ({2*\a*\s+\s}:\i){};
		\node[inner sep=1.5pt,circle,draw,fill=white] at ({2*\a*\s+\s}:\m){};
		\node[inner sep=1.5pt,circle,draw,fill] at ({2*\a*\s+\s}:\o){};
    \end{pgfonlayer}
    }
	\foreach \a in {2,...,\n}
	{
    \begin{pgfonlayer}{V}
		\node[inner sep=1.5pt,circle,draw,fill=white] at ({2*\a*\s}:\i){};
		\node[inner sep=1.5pt,circle,draw,fill] at ({2*\a*\s}:\m){};
		\node[inner sep=1.5pt,circle,draw,fill] at ({2*\a*\s}:\o){};
		\node[inner sep=1.5pt,circle,draw,fill] at ({2*\a*\s+\s}:\i){};
		\node[inner sep=1.5pt,circle,draw,fill=white] at ({2*\a*\s+\s}:\m){};
		\node[inner sep=1.5pt,circle,draw,fill] at ({2*\a*\s+\s}:\o){};
    \end{pgfonlayer}
    }
	\foreach \a in {1,...,\n}
	{
    \begin{pgfonlayer}{E}
    \foreach \x in {\i,\m,\o}
    {
    \path[draw] ({2*\a*\s}:\x)--({\s+2*\a*\s}:\x);
    \foreach \y in {\i,\m,\o}
    {
        \path[draw] ({\s+2*\a*\s}:\x)--({2*\s+2*\a*\s}:\y);
    }
    }
    \end{pgfonlayer}
    }
\end{tikzpicture}
\caption{Distinguishing colouring of $C_{6,K_{3,3}}$}
\label{fig:k33cycle}
\end{figure}

\begin{lemma}
\label{lem:specialgraphs}
The following graphs have distinguishing number at most $2$: \begin{enumerate}[label=(\arabic*),leftmargin=*]
\item \label{itm:linegraph-not-exceptional}
The line graph of every non-exceptional $3$-valent graph;
\item \label{itm:linegraph-pet-q3-wn}
The line graphs of the following graphs: the Petersen graph, $Q_3$, $K_3 \square K_3$, $K_5 \times K_2$, and $W_n$ for every $n \geq 3$;
\item \label{itm:bipcomphae}
The bipartite complement of the Heawood graph;
\item \label{itm:4hypercube}
The $4$-dimensional hypercube $Q_4$;
\item \label{itm:46cage}
The $(4,6)$-cage, and 
\item \label{itm:cycleofk33}
The graph $C_{n,K_{3,3}}$ for $n \geq 2$.
\end{enumerate}
\end{lemma}

\begin{proof}
Lemma~\ref{lem:linegraph} immediately implies \ref{itm:linegraph-not-exceptional}. 

For \ref{itm:linegraph-pet-q3-wn}, it suffices to observe that all the base graphs have at least $7$ vertices and a Hamiltonian path, and then apply Theorem~\ref{traceable} and Lemma~\ref{lem:linegraph} .

For \ref{itm:bipcomphae} note that the bipartite complement of the Heawood graph has the same automorphism group as the Heawood graph and thus also the same distinguishing number. By Theorem~\ref{theo:cubic}, this distinguishing number is $2$.

\ref{itm:4hypercube} follows from \cite{hypercubedistinguishing}, where distinguishing numbers of all hypercubes were determined.

For the proof of \ref{itm:46cage} first note that the $(4,6)$-cage is bipartite and any two vertices in each of its parts have exactly one neighbour in common. Let $v$ be any vertex, let $v_i$ for $1 \leq i \leq 4$ be the neighbours of $v$, and let $v_{ij}$ $1 \leq i \leq 3$ be the neighbours of $v_i$. 

Colour $v$ white, for $1 \leq i \leq 4$ colour $v_i$ black, and colour $v_{ij}$ black if $i < j$ and white otherwise. Finally colour the common neighbours of $v_{22}$ and $v_{32}$, and $v_{22}$ and $v_{33}$ black and all other vertices at distance $3$ from $v$ white, see Figure \ref{fig:46cage}.

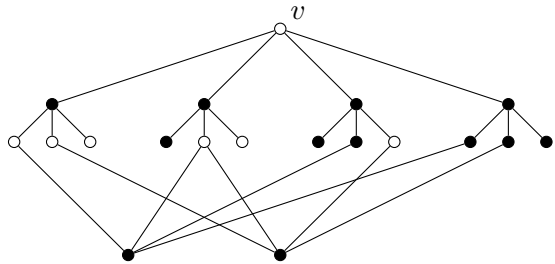
\begin{figure}
\centering
\begin{tikzpicture}
	\pgfdeclarelayer{E}
	\pgfdeclarelayer{V}
	\pgfsetlayers{E,V}
    \begin{pgfonlayer}{V}
		\node[inner sep=1.5pt,circle,draw,fill=white] (v) at (0,0) {};
    \node[above right] at (v) {$v$};
    \end{pgfonlayer}
    
	\foreach \a in {0,...,3}
	{
    \begin{pgfonlayer}{V}
		\node[inner sep=1.5pt,circle,draw,fill=black] (v\a) at (2*\a-3,-1) {};
    \end{pgfonlayer}
    \begin{pgfonlayer}{E}
        \path[draw] (v\a)--(v);
    \end{pgfonlayer}
    
	\foreach \b in {0,...,2}
	{\ifthenelse{\a>\b}{\colorlet{col}{black};}{\colorlet{col}{white};}
    \begin{pgfonlayer}{V}
		\node[inner sep=1.5pt,circle,draw,fill=col] (v\a\b) at (2*\a+.5*\b-3.5,-1.5) {};
    \end{pgfonlayer}
    \begin{pgfonlayer}{E}
        \path[draw] (v\a\b)--(v\a);
    \end{pgfonlayer}
    }
	}
    
    \begin{pgfonlayer}{V}
		\node[inner sep=1.5pt,circle,draw,fill=black] (w1) at (-2,-3) {};
		\node[inner sep=1.5pt,circle,draw,fill=black] (w2) at (0,-3) {};
    \end{pgfonlayer}
    
    \begin{pgfonlayer}{E}
        \path[draw] (w1)--(v00);
        \path[draw] (w1)--(v11);
        \path[draw] (w1)--(v21);
        \path[draw] (w1)--(v30);
        
        \path[draw] (w2)--(v01);
        \path[draw] (w2)--(v11);
        \path[draw] (w2)--(v22);
        \path[draw] (w2)--(v31);
    \end{pgfonlayer}
\end{tikzpicture}
\caption{Colouring of the $(4,6)$-cage, all vertices at distance $3$ from $v$ not shown in the picture are coloured white.}
\label{fig:46cage}
\end{figure}

Let $\gamma$ be a colour preserving automorphism. Then $\gamma$ must fix $v$, since it is the only white vertex with $4$ black neighbours. Furthermore $\gamma$ must fix all neighbours of $v$ since they have a different number of black neighbours. It must also fix the two black vertices at distance $3$ from $v$ for the same reason. Now it is easy to see that $\gamma$ has to fix all vertices at distance $2$ from $v$ and hence it is the identity.

For \ref{itm:cycleofk33}, consider the colouring shown in Figure~\ref{fig:k33cycle}. Note that the automorphism group has two orbits on edges: those that belong to a copy of $K_{3,3}$, and those that don't, which we call matching edges. There is a unique matching edge both of whose endpoints are coloured white. Every colour preserving automorphism must fix this edge and the matching it is contained in. The colours on the remaining edges in this matching make sure that every colour preserving automorphism must fix this matching pointwise, and thus must fix every matching between two copies of $K_{3,3}$ setwise. It is now easy to see that a colour preserving automorphism fixes all vertices of $C_{6,K_{3,3}}$. Finally note that this colouring can be generalised to a colouring of $C_{n,K_{3,3}}$ for any number $n \geq 2$. 
\end{proof}

\section{The proof of Theorem~\ref{thm:main}}

In this section, we prove our main result. Determining the distinguishing numbers of the exceptional graphs is straightforward and will be left to the reader.

To show that the remaining graphs have distinguishing number $2$, we distinguish cases according to the local group of $A:=\Aut G$. Define the \emph{type} of an edge $uv$ as the size of the orbit of $u$ under the action of the local group at $v$. By the orbit-stabiliser lemma, this is the index of $A_{uv}$ in $A_v$. Since by vertex transitivity $|A_v|=|A_u|$, this also shows that the type is well-defined, i.e.\ it does not depend on the endpoint of the edge. 

Note that since the orbits of the local group at $v$ partition the neighbourhood of $v$ the types of edges incident to $v$ correspond to a partition of $4$. Since $G$ is vertex-transitive, this partition is the same for every vertex. Since the only partitions of $4$ that do not contain a part of size $1$ are $(2,2)$ and $(4)$, we split up the proof of Theorem \ref{thm:main} into the following three cases:
\begin{enumerate}
\item There are edges of type $1$. This case is treated in Section \ref{sec:type1}.
\item All edges have type $2$. This is treated in Section~\ref{sec:type2}.
\item[4.] All edges have type $4$, and hence $G$ is arc-transitive. For this case, see Section~\ref{sec:arctransitive}.
\end{enumerate}

\subsection{Graphs with edges of type 1}
\label{sec:type1}

Let $G_{t \geq 2}$ be the graph obtained from $G$ by removing all edges of type $1$. Note that the components of $G_{t \geq 2}$ form a system of imprimitivity for $A$. 
We will need the following results.

\begin{lemma}
\label{lem:uniquetype1}
Assume that every vertex of $G$ is incident to a unique type $1$-edge, $G_{t \geq 2}$ is not connected, and any two components of $G_{t \geq 2}$ are connected by at most one type $1$-edge. Then $G$ has a distinguishing $2$-colouring.
\end{lemma}
\begin{proof}
Let $k$ be the number of vertices in a component of $G_{t \geq 2}$. Consider the graph $H$ obtained from $G$ by contracting every component of $G_{t \geq 2}$ to a single vertex. By our assumptions, $H$ is a $k$-regular graph and it follows from Lemma~\ref{degreebound} that its distinguishing number is at most $k+1$. Let $c'$ be a distinguishing colouring of $H$ with colours $\{0,1,\dots,k\}$. We now colour $G$ in the following way: in every component of $G_{t \geq 2}$, we colour as many vertices black as the colour of the corresponding vertex of $H$ suggests. 
 
Since $c'$ is distinguishing, any automorphism which preserves the resulting colouring has to fix all components of $G_{t \geq 2}$ setwise. As every type 1 edge is uniquely identified by the components it connects, each type 1 edge and hence also every vertex must be fixed by every colour-preserving automorphism.
\end{proof}

\begin{lemma}
\label{lem:deletetype1}
Let $G$ be a connected vertex-transitive graph. Assume that $G_{t \geq 2}$ is not connected, let $H$ be a component of $G_{t \geq 2}$ and let $v \in H$. If $H$ admits a $2$-colouring $c'$ such that the only automorphism of $H$ fixing $v$ and preserving $c'$ is the identity, then  $G$ has a distinguishing $2$-colouring.
\end{lemma}

\begin{proof}
Denote the components of $G_{t \geq 2}$ by $H_1, \dots, H_R$. Note that each $H_i$ is isomorphic to $H$. Let $v_1 \in H_1$. Note that the graph obtained from $G$ by contracting the components $H_1, \dots, H_R$ is connected and vertex-transitive and thus at least $2$-connected. Hence $G - H_1$ is connected, and thus $(G - H_1) + v_1$ is connected as well. 

For $i\in\{2,\ldots,R\}$, pick some shortest path from $H_i$ to $v_1$ in $(G - H_1) + v_1$ and let $v_i$ and $e_i$ be the first vertex and edge of this path, respectively. Without loss of generality we may assume that the number of black vertices in $c'$ is not exactly one---otherwise change the colour of $v$ to obtain a colouring with this property. Let $\pi_i \colon H \to H_i$ be an isomorphism which maps $v$ to $v_i$. Such an isomorphism exists because $G$ (and thus also $H$) is vertex-transitive. Now define a colouring $c$ of $G$ by
\[
c(x) =
\begin{cases}
\mathrm{black} & \mathrm{if}\; x = v_1,\\
\mathrm{white} & \mathrm{if}\; x \in H_1 - v_1,\\
c'(\pi_i^{-1}(x)) & \mathrm{if}\; x \in H_i \;\mathrm{for}\; i \neq 1.
\end{cases}
\]
Let $\gamma$ be an automorphism of $G$ preserving $c$. We show that $\gamma$ fixes every vertex and thus $c$ is distinguishing.

First, note that $\gamma$ must fix $v_1$, since $v_1$ is the only black vertex in $H_1$ which in turn is the only component with a unique black vertex. 

Next we show that, for $i \neq 1$, every $H_i$  must be fixed pointwise by $\gamma$. Assume not.  Let $H_i$ be a component such that the distance from $H_i$ to $v_1$ is minimal, among the components that are not fixed pointwise. The endpoint $u_i$ of $e_i$ which does not lie in $H_i$ is either $v_1$, or it lies in some component $H_j$ which is closer to $v_1$. Hence $u_i$ is fixed by $\gamma$. Since $e_1$ has type $1$, $\gamma$ must also fix $v_i$ and thus induce an automorphism of $H_i$. By hypothesis, this induced automorphism is trivial and thus $\gamma$ fixes $H_i$ pointwise. 

Finally, let $x \in H_1 - v_1$. Then $x$ is incident to an edge of type 1 which connects $H_1$ to a different component $H_i$. Since the other endpoint of this edge is fixed by $\gamma$, the same must be true for $x$.
\end{proof}

\begin{corollary}
\label{cor:deletetype1}
Let $G$ be a connected, vertex-transitive graph and let $H$ be a component of $G_{t \geq 2}$. If $H$ has a distinguishing $2$-colouring, then so does $G$.
\end{corollary}

\begin{proof}
If $H$ is the only component of $G_{t \geq 2}$, then a distinguishing colouring of $H$ is also distinguishing for $G$, otherwise apply Lemma \ref{lem:deletetype1}.
\end{proof}

\begin{theorem}
\label{thm:type1edges}
Let $G$ be a connected $4$-valent vertex-transitive graph containing edges of type $1$. Then $D(G) = 2$, unless $G$ is $K_4 \square K_2$. 
\end{theorem}

\begin{proof}
If all edges are of type $1$, then $A_v=1$ and thus  colouring one vertex black and all other vertices white yields a distinguishing colouring.

Next assume that the local group has two orbits of size $1$ and one orbit of size $2$. In this case $G_{t \geq 2}$ is a union of cycles. If there is only one such cycle, then it must have length $6$ or more, and hence $G$ is $2$-distinguishable by Corollary \ref{cor:deletetype1}. If there is more than one, then the conditions of Lemma \ref{lem:deletetype1} are satisfied.

Finally consider the case where the local group has one orbit of size $1$ and one orbit of size $3$. All components of $G_{t \geq 2}$ are isomorphic to some $3$-regular vertex-transitive graph $G'$. Also note that the induced action of $A$ on $G'$ is arc-transitive.

If $G'$ has distinguishing number $2$, then we can apply Corollary \ref{cor:deletetype1} to obtain a distinguishing $2$-colouring of $G$. By Theorem~\ref{theo:cubic}, the only other possibility is that $G'$ is isomorphic to one of  $K_4, K_{3,3}, Q_3$ or the Petersen graph.

If $G_{t \geq 2}$ is connected, then $G$ is obtained from $G'$ by adding edges of type $1$. Since $A$ is arc-transitive on $G'$, no edge of type $1$ can connect two neighbours (in $G'$) of the same vertex. Otherwise any two neighbours of this vertex would have to be connected by an edge, contradicting the fact that each vertex of $G$ is adjacent to only one edge of type $1$. Hence an edge of type $1$ can't connect vertices at distance at most $2$ in $G'$. This rules out $K_4, K_{3,3}$ and the Petersen graph as possibilities for $G'$, since they have diameter at most $2$. The only way to add edges with respect to this constraint in the cube $Q_3$ yields $G=K_{4,4}$ which does not contain edges of type $1$.

Thus we can assume that $G_{t \geq 2}$ is not connected. Both the Petersen graph and $Q_3$ have colourings satisfying the condition of Lemma \ref{lem:deletetype1}, see Figure~\ref{fig:cubepetersenvertexstabiliser}. Hence if $G'$ is one of them, then $G$ has a distinguishing $2$-colouring.

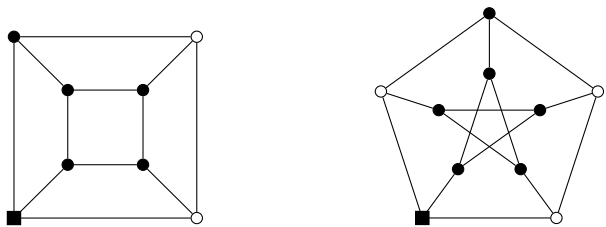
\begin{figure}
\centering
\begin{tikzpicture}
	\pgfdeclarelayer{E}
	\pgfdeclarelayer{V}
	\pgfsetlayers{E,V}
	\tikzmath{\i=.7; \o=1.7; \n=4; \s=360/\n;}
    \begin{pgfonlayer}{V}
	\foreach \a in {1,...,2}
	{
		\node[inner sep=1.5pt,circle,draw,fill] (v\a) at ({\a*\s+45}:\o){};
		\node[inner sep=1.5pt,circle,draw,fill] (w\a) at ({\a*\s+45}:\i){};
	}
	\foreach \a in {3,4}
	{
		\node[inner sep=1.5pt,circle,draw,fill=white] (v\a) at ({\a*\s+45}:\o){};
		\node[inner sep=1.5pt,circle,draw,fill] (w\a) at ({\a*\s+45}:\i){};
	}
    \node[inner sep=2.5pt,rectangle,draw,fill] (v) at (v2){};
    \end{pgfonlayer}

    \begin{pgfonlayer}{E}
        \path[draw] (v1)--(w1);
        \path[draw] (v2)--(w2);
        \path[draw] (v3)--(w3);
        \path[draw] (v4)--(w4);
        
        \path[draw] (v1)--(v2);
        \path[draw] (v2)--(v3);
        \path[draw] (v3)--(v4);
        \path[draw] (v4)--(v1);
        
        \path[draw] (w1)--(w2);
        \path[draw] (w2)--(w3);
        \path[draw] (w3)--(w4);
        \path[draw] (w4)--(w1);
    \end{pgfonlayer}
\end{tikzpicture}
\hspace{2cm}
\begin{tikzpicture}
	\pgfdeclarelayer{E}
	\pgfdeclarelayer{V}
	\pgfsetlayers{E,V}
	\tikzmath{\i=.7; \o=1.5; \n=5; \s=360/\n;}
    \begin{pgfonlayer}{V}
	\foreach \a in {1,...,4}
	{
		\node[inner sep=1.5pt,circle,draw,fill=white] (v\a) at ({\a*\s+90}:\o){};
		\node[inner sep=1.5pt,circle,draw,fill] (w\a) at ({\a*\s+90}:\i){};
     }
	\foreach \a in {5}
	{
		\node[inner sep=1.5pt,circle,draw,fill] (v\a) at ({\a*\s+90}:\o){};
		\node[inner sep=1.5pt,circle,draw,fill] (w\a) at ({\a*\s+90}:\i){};
     }
    \node[inner sep=2.5pt,rectangle,draw,fill] (v) at (v2){};
    \end{pgfonlayer}

    \begin{pgfonlayer}{E}
        \path[draw] (v1)--(w1);
        \path[draw] (v2)--(w2);
        \path[draw] (v3)--(w3);
        \path[draw] (v4)--(w4);
        \path[draw] (v5)--(w5);
        
        \path[draw] (v1)--(v2);
        \path[draw] (v2)--(v3);
        \path[draw] (v3)--(v4);
        \path[draw] (v4)--(v5);
        \path[draw] (v5)--(v1);
        
        \path[draw] (w1)--(w3);
        \path[draw] (w2)--(w4);
        \path[draw] (w3)--(w5);
        \path[draw] (w4)--(w1);
        \path[draw] (w5)--(w2);
    \end{pgfonlayer}
\end{tikzpicture}
\caption{Colourings satisfying the condition of Lemma \ref{lem:deletetype1}, $v$ is the square vertex.}
\label{fig:cubepetersenvertexstabiliser}
\end{figure}

We may thus assume that  $G'$ is either $K_4$ or $K_{3,3}$. By Lemma \ref{lem:uniquetype1} we may assume that there is a pair of components of $G_{t \geq 2}$ connected by multiple type $1$ edges. Since $G$ is vertex-transitive and each vertex is incident to a unique edge of type 1, the number of type 1 edges between any pair of adjacent components of $G_{t \geq 2}$ is independent of the choice of the pair. Furthermore, recall that $A$ acts arc-transitively on $G'$. Hence if two adjacent vertices in a component $H$ are both adjacent to the same component $H'$ (via type 1 edges), then all vertices of $H$ are adjacent to $H'$. For $G' = K_4$, this is the only possibility, and the resulting graph is $G=K_4 \square K_2$. For $G' = K_{3,3}$, the above observation tells us that all vertices in the same bipartite class of a component send their type 1 edges to the same component, and hence $G = C_{n,K_{3,3}}$ (see Figure~\ref{fig:k33cycle}) for some $n\geq 2$, which has distinguishing number $2$.
\end{proof}

\subsection{Graphs with only edges of type 2}
\label{sec:type2}
In this section, we assume that all edges of $G$ are of type $2$. This implies that $A$ has two orbits on arcs and therefore at most two orbits on edges. We distinguish two subcases according to whether $G$ is edge-transitive or not.

\subsubsection{Edge-transitive case}
\begin{theorem}
Let $G$ be a connected $4$-valent graph that is vertex- and edge-transitive but not arc-transitive. Then $D(G) = 2$.
\end{theorem}
\begin{proof}
In this case, $A$ has two orbits on arcs and each arc is in a different orbit than its inverse arc. By removing one of the two orbits, $G$ becomes an arc-transitive directed graph in which every vertex has in- and out-degree $2$.  There is some $s \geq 1$ such that $A$ acts regularly on directed $s$-arcs (see for example \cite[Lemma 5.4(v)]{PotVer}). 

Let $P= (v_0, \dots, v_s)$  be a directed $s$-arc in $G$. Suppose for a contradiction that there is an arc from $v_s$ to $v_0$. Clearly, in this case $s \geq 2$, as $G$ does not contain any $2$-cycles. There is an automorphism fixing $(v_0,\ldots,v_{s-1})$ pointwise, but not fixing $v_s$. Therefore, the second out-neighbour $v'_s \neq v_s$ of $v_{s-1}$ must also have $v_0$ as an out-neighbour. By directed $2$-arc-transitivity we conclude that for any vertex $v_i$ on $P$, the out-neighbours of $v_i$ are exactly the in-neighbours of $v_{i+2}$, so the digraph is a directed wreath graph and $G$ is arc-transitive, which gives the desired contradiction.

We may thus assume that there is no arc from $v_s$ to $v_0$. Colour the vertices of $P$ black and the remaining vertices white. Note that $v_0$ is the unique black vertex with no black in-neighbour. Hence $v_0$ and thus all of $P$ must be fixed by any colour-preserving automorphism. By $s$-arc-regularity, this implies that the colouring is distinguishing and $G$ has distinguishing number $2$. 
\end{proof}

\subsubsection{Non-edge-transitive case}
If $G$ is not edge-transitive, then there must be $2$ orbits on edges each of which forms a disjoint union of cycles. Denote the two subgraphs induced by the edge orbits by $G_1$ and $G_2$. By transitivity, all cycles in $G_1$ have the same length, the same is true for $G_2$.

We will inductively construct a distinguishing colouring from partial colourings of $G$. Let $\tilde c$ be a partial colouring of $G$ with domain $\tilde V \subseteq V$, that is, $\tilde c$ is a function from $\tilde V$ to some set $C$ of colours. An \emph{extension} of $\tilde c$ is a colouring $c$ of $G$ such that $c$ and $\tilde c$ coincide on $\tilde V$.

\begin{lemma}
\label{lem:cyclecolourinduction}
Let $G$ be a connected $4$-valent vertex-transitive but not edge-transitive graph and assume that all edges have type $2$. Let $G_1$ and $G_2$ be the subgraphs induced by the two edge orbits. Let $V'$ be a set of vertices of $G$ and let $C$ be a cycle in $G_1$ which is disjoint from $V'$ and contains a neighbour $v$ of some vertex in $V'$. Then there is a cycle $D$ in $G_1$ which is disjoint from $V'$ (possibly $D = C$) and a partial $2$-colouring $\tilde c$ of $G$ with domain $C \cup D$ such that
\begin{itemize}
\item $C$ and $D$ both contain either $1$ or $2$ black vertices, and
\item if $\gamma \in \Aut G$ fixes $V'$ pointwise and fixes any extension of $\tilde c$, then $\gamma$ fixes $V' \cup C \cup D$ pointwise.
\end{itemize}
\end{lemma}

\begin{proof}
Call a vertex $u$ a twin of $v$ if there is an automorphism in the stabiliser of $V'$ that moves $u$ to $v$. Note that $v$ has at most one twin, since there is an edge in $G_2$ connecting $v$ to some $w$ in $V'$, and $w$ has only one other neighbour in $G_2$.

If $v$ has no twin then every automorphism that fixes $V'$ pointwise must fix $v$. Set $D=C$, colour $v$ and one of its neighbours on $C$ black and colour the remaining vertices of $C$ white. Then every automorphism which fixes $V'$ as well as an  extension of this colouring must fix $v$ and its black neighbour and thus also fixes $C$.

Next assume that $v$ has a twin that lies on $C$. Again let $D = C$ and colour $v$ and one of its neighbours in $C$ black, but make sure that the black neighbour of $v$ is not a twin of $v$. The same argument as above tells us that this colouring has the desired properties.

Finally assume that $v$ has a twin $u$ that lies outside of $C$. Let $D$ be the cycle in $G_1$ containing $u$ and observe that $D$ is also disjoint from $V'$. Colour $v$ and one of its neighbours in $C$ black, colour one of the neighbours of $u$ in $D$ black, and colour the remaining vertices of $C \cup D$ white. Any automorphism that fixes $V'$ as well as an extension of this colouring must fix $u$ and $v$ and their respective black neighbours, whence we have found the desired colouring.
\end{proof}

\begin{theorem}
\label{thm:type2edges}
Let $G$ be a connected $4$-valent vertex-transitive but not edge-transitive graph and assume that all edges have type $2$. Then $D(G) = 2$.
\end{theorem}

\begin{proof}
 Let $G_1$ and $G_2$ be the subgraphs induced by the two edge orbits respectively and without loss of generality assume that cycles in $G_1$ are at least as long as cycles in $G_2$.

If $G_1$ consists of a single cycle then this cycle must have length at least $6$. Hence there is a distinguishing $2$-colouring of $G_1$ which must also be distinguishing $2$-colouring of $G$. Hence we may assume that $G_1$ consists of more than one cycle.

If cycles in $G_1$ have length at least $4$, then let  $C_1$ be a cycle in $G_1$ and let $v_1$ be a vertex on this cycle. Now inductively apply Lemma~\ref{lem:cyclecolourinduction}. For the first step, let $V' = \{v_1\}$. In each step, pick a cycle $C \neq C_1$ which contains a $G_2$-neighbour of $V'$, colour it according to the lemma and add the vertices of $C \cup D$ to $V'$. The graph obtained from $G$ by contracting every cycle in $G_1$ is connected and vertex-transitive. Hence, by Lemma \ref{lem:connectivity} it is $2$-connected and remains connected after removing $C_1$. In particular, the above colouring procedure assigns colours to all vertices except those in $C_1$. Finally colour $v_1$ and its neighbours on $C_1$ black, and colour the rest of $C_1$ white.

We claim that the resulting colouring is distinguishing. Clearly, every colour-preserving automorphism must fix $v_1$ since it is the only black vertex both of whose neighbours in $G_1$ are black (recall that $C_1$ is the only cycle in $G_1$ containing $3$ black vertices). Using Lemma~\ref{lem:cyclecolourinduction} inductively, we see that every colour-preserving automorphism must  fix every cycle pointwise, except possibly $C_1$. Hence the colouring is distinguishing unless the two neighbours of $v_1$ in $G_1$ have the same $G_2$-neighbourhood. In this case, by vertex-transitivity any two vertices at distance $2$ in $G_1$ have the same $G_2$-neighbourhood. If cycles in $G_1$ have length $5$ or more, this implies that vertices have degree at least $3$ in $G_2$  which is a contradiction. If cycles in $G_1$ have length $4$, then so do cycles in $G_2$ and $G$ is a graph obtained by identifying antipodal points of $4$-cycles, i.e., a wreath graph, which contradicts the assumption that $G$ is not edge transitive.

It remains to deal with the case when both $G_1$ and $G_2$ are disjoint unions of $3$-cycles.
Let $H$ be the graph with vertices these $3$-cycles, with two such $3$-cycles being adjacent in $H$ if they share a vertex in $G$. It is easy to see that $H$ is regular of valency 3 and $G=L(H)$. By Theorem~\ref{monikabound}, we have $D(G)=D'(H)\leq 2$, unless $H$ is $K_4$ or $K_{3,3}$. Finally, note that $L(K_4)\cong W_3$ while $L(K_{3,3})\cong K_3 \square K_3$.

\end{proof}

\subsection{Arc-transitive graphs}
\label{sec:arctransitive}

We first prove a few lemmas to show that we can restrict ourselves to graphs with girth~$4$.

\begin{lemma}\label{lem:girth3}
Let $G$ be a connected $4$-valent arc-transitive graph. If $G$ has girth $3$, then $G$ is either $K_5$ or $W_3$, or the line graph of a $3$-valent arc-transitive graph.
\end{lemma}

\begin{proof}
Follows from ~\cite[Theorem~5.1(1)]{Girth4}).
\end{proof}

\begin{lemma}
\label{lem:girth-46cage}
Let $G$ be a connected graph of minimal valency at least $3$ and girth $g \geq 5$. If $G$ is $s$-arc-transitive, then $s \leq g-3$, unless $G$ is a Moore graph of girth $5$, or the incidence graph of a projective plane.
\end{lemma}

\begin{proof}
Assume for a contradiction that $G$ is $(g-2)$-arc-transitive. Let $C = (v_0, \dots, v_{g-1})$ be a cycle of length $g$. Note that $(v_0,\dots,v_{g-2})$ is a $(g-2)$-arc and that its endpoints have a common neighbour. By $(g-2)$-arc-transitivity, every $(g-2)$-arc has this property.

Let $v_{g-2}'$ be a neighbour of $v_{g-3}$ outside of $C$. Then $(v_0, \dots , v_{g-3}, v_{g-2}')$ is a $(g-2)$-arc, whence $v_{g-2}'$ and $v_0$ have a common neighbour $v_{g-1}'$. Now the closed walk $(v_0,v_{g-1},v_{g-2},v_{g-3},v_{g-2}',v_{g-1}',v_0)$ shows that $g \leq 6$.

If $g=5$, then the fact that the endpoints of every $3$-arc have a common neighbour implies that $G$ has diameter $2$ and is thus a Moore graph.

If $g=6$, then an analogous argument as above yields that $G$ has diameter $3$. If $G$ was not bipartite, then for $v\in V$ there would be an edge connecting two vertices $x$ and $y$ at the same distance from $v$, and since $g=6$ we have $d(x,v) = d(y,v) = 3$. But then there is a $4$-arc from $v$ to $x$ whence by the above argument $v$ and $x$ have a common neighbour, contradicting $d(x,v)=3$. 

Hence $G$ is bipartite and every vertex at distance $2$ from a given vertex $v$ has a unique common neighbour with $v$. It follows that $G$ is the incidence graph of a projective plane.
\end{proof}

\begin{lemma}
\label{lem:girth5d2}
Let $G$ be a connected $4$-valent arc-transitive graph of girth at least $5$, then $D(G) = 2$.
\end{lemma}

\begin{proof}
Let $g$ be the girth of $G$ and let $s$ be such that $G$ is $s$-arc-transitive but not $(s+1)$-arc-transitive. Note that there is no $4$-valent Moore graph, and that there is a unique $4$-valent graph that is the incidence graph of a projective plane, namely the $(4,6)$-cage. By Lemmas~\ref{lem:specialgraphs} and \ref{lem:girth-46cage} we may thus assume that $s \leq g - 3$.

By Lemma \ref{lem:cyclicconnectivity}, there is a cycle $C = (v_0, \dots, v_{g-1})$ such that $G-C$ is $2$-edge connected. Let $P=(v_{s+1},v_{s},\ldots,v_1)$ and let $X$ be its pointwise stabiliser. Note that $P$ is an $s$-arc and thus $X$ is not transitive on $N(v_1)\setminus \{v_2\}$ (otherwise $G$ would be $(s+1)$-arc-transitive). Let $v_0'$ be a neighbour of $v_1$ that is in a different orbit than $v_0$ under $X$.

Note that the subgraph induced by the vertices $\{v_0',v_0,v_1,\dots, v_{g-2}\}$ is a tree since any additional edge between these vertices would give a cycle of length less than $g$. Denote this tree by $T$ and let $H$ be the subgraph obtained from $G$ by removing all vertices of $T$. Observe that $v_0'$ has degree at most $3$ in $G-C$. If $H$ is not connected, then there is one component of $H$ that is connected to $v_0'$ by a unique edge. Removing that edge from $G-C$ would disconnect it, contradicting the fact that $G-C$ is $2$-edge connected. It follows that $H$ is connected.

Colour all vertices of $T$ black and colour $v_{g-1}$ white. Inductively colour the vertices of $G$ as follows: Let $x$ be a vertex at minimal distance to $v_{g-1}$ in $H$ that has not been coloured yet. If $x$ is fixed by the pointwise stabiliser in $A$ of all previously coloured points, then colour it white. Otherwise colour it black.

We claim that this colouring is distinguishing. First note that if an automorphism fixes two neighbours $u$ and $w$ of a vertex $v$, then it must also fix $v$, since otherwise the image of $v$ would also be a common neighbour of $u$ and $w$ contradicting $g \geq 5$. Note that this implies that all vertices in $H$ with a neighbour outside of $H$ are coloured white. Indeed, at the time such a vertex $x$ is considered for colouring, two of its neighbours are already coloured: its predecessor on a shortest $v_{g-1}$-$x$-path in $H$ and its neighbour outside of $H$. Hence by the previous observation, $x$ is coloured white. 

Next we show that $v_1$ is the only black vertex with three black neighbours. By the above observations it is the only such vertex in $T$. Now let $x$ be a black vertex in $H$. Then at most one neighbour of $x$ was coloured before $x$ (otherwise we would have coloured $x$ white). Furthermore, if $P$ is a shortest $v_{g-1}$-$x$-path in $H$, then $P \cup C$ contains an $s$-arc ending in $x$. Hence the pointwise stabiliser of $x$ and all vertices coloured before $x$ does not act transitively on the remaining neighbours of $x$, whence at most one of them will be coloured black.

Let $\gamma$ be a colour preserving automorphism. The above discussion shows that $\gamma$ must fix $v_1$. Furthermore all neighbours of $T$ are white, so $\gamma$ must preserve $T$ setwise. Since there is no automorphism of $G$ that fixes $(v_1,\dots,v_{g-2})$ and moves $v_0$ to $v_0'$, $\gamma$ must fix $T$ pointwise. Finally assume that there is a vertex in $H$ that is not fixed by $\gamma$ and let $x$ be the first such vertex that was coloured in the inductive procedure. Clearly, $x$ is coloured black. Let $y$ be the neighbour of $x$ on a shortest $v_{g-1}$-$x$-path $P$, and let $S$ be an $s$-arc contained in $C \cup P$. Then $S$ is pointwise stabilised by $\gamma$, and since the orbit of $x$ under the pointwise stabiliser of $S$ is not a singleton, it contains exactly one other element $x'$. Every automorphism that fixes $x$ and $S$ also fixes $x'$ and vice versa. Hence at most one of $x$ and $x'$ can be coloured black and thus neither of them can be moved by $\gamma$.
\end{proof}
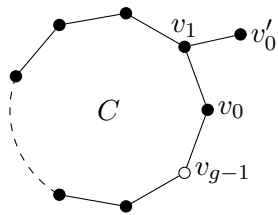
\begin{figure}
\centering
\begin{tikzpicture}
	\pgfdeclarelayer{E}
	\pgfdeclarelayer{V}
	\pgfsetlayers{E,main,V}
	\tikzmath{\i=1.3; \o=2; \n=9; \s=360/\n;}
    
    \begin{pgfonlayer}{V}
	\foreach \a in {1,...,5,7,8}
	{
		\node[inner sep=1.5pt,circle,draw,fill] (v\a) at ({\a*\s-40}:\i){};
    }
    
	\node[inner sep=1.5pt,circle,draw,fill=white] (v9) at (-40:\i){};
    \node[inner sep=1.5pt,circle,draw,fill] (u2) at (30:\o){};
    \end{pgfonlayer}
    
    \begin{pgfonlayer}{E}
	\foreach \a in {-2,...,4}
    {
        \path[draw] ({\a*\s-40}:\i)--({\s+\a*\s-40}:\i);
	}
    \path[draw] (u2)--(v2);
    \draw[dashed] ({5*\s-40}:\i) arc ({5*\s-40}:{7*\s-40}:\i);
    \end{pgfonlayer}
    \node at (0,0) {$C$};
    \node[right] at (v1) {$v_0$};
    \node[above] at (v2) {$v_1$};
    \node[right] at (v9) {$v_{g-1}$};
    \node[right] at (u2) {$v_0'$};
\end{tikzpicture}
\caption{The tree $T$ in the proof of Lemma \ref{lem:girth5d2}.}
\label{fig:girth5d2}
\end{figure}

Next we give some results for the case when $G$ has girth exactly $4$. Note that in this case, there must be vertices at distance $2$ from each other with $2$ or more common neighbours. The following two lemmas follow from results in~\cite{Girth4}.

\begin{lemma} \label{lem:girth4manyneighbours}
Let $G$ be a connected $4$-valent arc-transitive graph. If there are two vertices at distance $2$ with $3$ or more common neighbours, then $G$ is isomorphic to either $K_5\times K_2$ or $W_n$ for some $n \geq 4$.
\end{lemma}

\begin{proof}
If there are vertices with $4$ common neighbours, then by \cite[Lemma 4.3]{Girth4}, $G$ is a wreath graph. Otherwise, Subcase II.A of the proof  \cite[Theorem 3.3]{Girth4} implies that $G\cong K_5\times K_2$.
\end{proof}

\begin{lemma}\label{2ATGirth4}
Let $G$ be a connected  $4$-valent $2$-arc-transitive graph. If $G$ has girth $4$ but no two vertices at distance $2$ have more than $2$ common neighbours, then $G$ is isomorphic to either $Q_4$, or the bipartite complement of the Heawood graph.
\end{lemma}

\begin{proof}
By $2$-arc-transitivity, every edge is contained in at least three $4$-cycles. Subcase II.B of the proof of \cite[Theorem 3.3]{Girth4} then implies that $G$ is isomorphic to one of the two graphs as claimed.
\end{proof}

The hardest case to deal with is when the graph is locally $D_4$. In this case,  we take advantage of the following structural property. Note that $D_4$ in its natural action on $4$ points admits a unique system of imprimitivity with $2$ blocks of size $2$. We say that a $2$-arc $(v_0,v_1,v_2)$ is \emph{straight}, if $\{v_0,v_2\}$ is a block with respect to the local group at $v_1$, and  \emph{crooked} otherwise. 
Note that, of the three $2$-arcs starting with a given arc, one is straight and two are crooked. Further note that fixing a crooked $2$-arc fixes all neighbours of its midpoint. Finally, note that $A$ acts transitively on crooked $2$-arcs of $G$. Call a cycle in $G$ \emph{straight}, if all sub-arcs of length $2$ are straight.

\begin{theorem}
\label{thm:arctransitive}
Let $G$ be a connected $4$-valent arc-transitive graph, then $D(G) = 2$ unless $G$ is $K_5$, $K_3\square K_3$, $K_5 \times K_2$, or $W_n$ for some $n \geq 3$.
\end{theorem}

\begin{proof}
By Lemmas \ref{lem:girth3}, \ref{lem:girth5d2}, as well as Lemma \ref{lem:specialgraphs}, we can assume that $G$ has girth $4$.  By Lemma \ref{lem:girth4manyneighbours}, we can assume that no two vertices have more than two common neighbours.

Since $G$ is arc-transitive, the local group must be a transitive subgroup of $S_4$. If the local group is $2$-transitive, then $G$ is $2$-arc-transitive and this case is handled with Lemmas \ref{2ATGirth4} and \ref{lem:specialgraphs}. 

If the local group is $C_4$ or $V_4$, then $G$ is arc-regular. One can then colour one vertex $v$ and three of its neighbours black, and colour the remaining vertices white. Any colour preserving automorphism must fix the arc from $v$ to its unique white neighbour, thus the colouring is distinguishing. 

The last remaining case is that $G$ is locally $D_4$. Suppose first that $G$ contains a $4$-cycle that is not straight. Let $(u,v,w,x)$ be a $4$-cycle of $G$ such that $(u,v,w)$ is a crooked $2$-arc. 

We claim that any automorphism fixing $u$ and all of its neighbours must be the identity. By arc-transitivity and connectedness it is enough to show that such an automorphism must fix all neighbours of $v$. Since no pair of vertices has more than two common neighbours, $u$ and $w$ are the only two common neighbours of $v$ and $x$. In particular, if an automorphism fixes $w$ and all its neighbours, then it must also fix $u$. Hence it fixes a crooked $2$-arc with midpoint $v$, and thus it fixes $v$ and all of its neighbours, thus proving our claim.

Let $y$ be the unique vertex such that $(v,w,y)$ is a straight $2$-arc, and let $P=(u,v,w,y)$. Suppose that $y$ is adjacent to $u$. Let $u'$ be the unique vertex other than $u$ such that $(u',v,w)$ is crooked. Note that there is an automorphism fixing $v$ and $w$ (and thus $y$) and mapping $u$ to $u'$, and thus $y$ is adjacent to $u'$, and $v$ and $y$ have at least $3$ common neighbours ($u$, $u'$, and $w$), contradicting an earlier hypothesis. We conclude that $y$ is not adjacent to $u$ and thus the induced subgraph on $P$ is a path of length $3$. Colour $P$ black and colour the remaining vertices white. Since $(u,v,w)$ is crooked, but $(v,w,y)$ is straight, every colour preserving automorphism fixes $P$ pointwise, and thus it fixes $v$ and all its neighbours. Hence, by the above claim, this colouring is distinguishing.

From now on, we can assume that all $4$-cycles of $G$ are straight. Let $\mathcal C$ be the set of all $4$-cycles. Note that every edge is contained in a unique straight $4$-cycle, whence $\mathcal C$ forms a partition of $E(G)$. Furthermore, any two elements of $\mathcal C$ intersect in at most one vertex, since otherwise there would be vertices with $3$ or more common neighbours.

Now consider the auxiliary graph $G'$ with vertex set $\mathcal C$ and an edge between two vertices if the $4$-cycles have a vertex in common. Note that $G'$ is a $4$-valent graph  on $|\mathcal C| = \frac{|E(G)|}4 = \frac{|V(G)|}2$ vertices.

Note that $A$ has a natural induced action on $G'$, and this is easily seen to be locally $D_4$. Furthermore any distinguishing colouring of $L(G')$ corresponds to a distinguishing colouring of $G$. By Lemma \ref{lem:linegraph} and the above observations $D(G') \geq D(L(G')) \geq D(G)$. Hence if $D(G') = 2$, then $D(G) = 2$ and we are done. By induction, we may thus assume that $G'$ is one of $K_5$, $K_3\square K_3$, $K_5 \times K_2$, or $W_n$ for some $n \geq 3$. If $G'\neq K_5$, then by Lemma~\ref{lem:specialgraphs} \ref{itm:linegraph-pet-q3-wn}, we have $D(L(G'))=2$ and we are done. Finally note that $G' = K_5$ is not possible, since $A$ induces a transitive, locally $D_4$ action on $G'$, but $K_5$ admits no such action.
\end{proof}

\noindent\textsc{Acknowledgements.}
We would like to thank the anonymous referees for a number of helpful suggestions.

\bibliographystyle{abbrv}
\bibliography{bibliography.bib}
\end{document}